\documentclass[12pt]{amsart}
\usepackage{amsmath,amsthm,amssymb}

\newcommand{\fr}{\mathcal{F}}

\newcommand{\C}{\mathbb{C}}

\numberwithin{equation}{section}
\newtheorem{theorem}{Theorem}[section]
\newtheorem{lemma}[theorem]{Lemma}

\newtheorem*{thmA}{Theorem A}
\newtheorem*{thmC}{Theorem C}
\newtheorem*{thmD}{Theorem D}
\newtheorem*{thmB}{Theorem B}

\theoremstyle{remark}

\makeatletter
\@namedef{subjclassname@2010}{%
  \textup{2010} Mathematics Subject Classification}
\makeatother

\begin{document}

\title[Normality Criteria for Families of Meromorphic Functions about Shared Functions ]{Normality Criteria for Families of Meromorphic Functions about Shared Functions}

\author[S. Kumar	]{Sanjay Kumar}
\address{Department of Mathematics, Deen Dayal Upadhyaya College, University of Delhi,
Delhi--110 078, India }
\email{skpant@ddu.du.ac.in}
\author[P. Rani]{Poonam Rani}
\address{Department of Mathematics, University of Delhi,
  Delhi--110 007, India} \email{pnmrani753@gmail.com}

\begin{abstract}
In this paper we prove some normality criteria for a family of meromorphic functions concerning shared analytic functions, which extend or generalized some result obtained by Y. F. Wang,  M. L. Fang~\cite{WF} and  J. Qui, T. Zhu ~\cite{QZ}.
\end{abstract}


\maketitle
\section{Introduction and main results}
Let $D$ be a domain in $\mathbb C$, and $\mathcal  F$ be a family of meromorphic function in a domain $D$.  $ \mathcal F$ is said to be normal in a domain $D$, in the sense of Montel, if for each sequence $\{f_j\} \in \mathcal F$  there exist a subsequence $\{f_{j_k}\}$, such that  $\{f_{j_k}\}$ converges spherically locally uniformly on $D$, to a meromorphic function or $\infty$ \cite{Ahl, Hay, Schiff, Yang}.\\

Wilhelm Schwick ~\cite{Sch} was the first who gave a connection between normality and shared values and  proved a theorem   which says that:  {\it{A family $\fr$ of meromorphic functions on a domain ${D}$ is normal, if $f$ and $f'$ share $a_1$, $a_2$, $a_3$ for every $f\in \mathcal F$, where $a_1$, $a_2$, $a_3$ are distinct complex numbers.}}

Let us recall the definition of shared value. Let $f$ be a meromorphic function of a domain $D\subset\C$. For $p\in \C$, let
\begin{equation*}
  E_f(p)=\{z\in D: f(z)=p\}
\end{equation*}
and let
\begin{equation*}
E_f(\infty)= \text{poles of\ } f \text{\ in\ } D.
\end{equation*}
 For $p\in \C\cup \{\infty\}$, two meromorphic functions $f$ and $g$ of $D$ share the value $p$ if $E_f(p)=E_g(p).$\\

In 1998, Wang and Fang ~\cite{WF} obtained the following result.
\begin{thmA}\label{theorem1}
Let $\mathcal F$ be a family of meromorphic functions in a domain $D.$  Let $k$ be a positive integer and $b$ be a non zero finite complex number. If for each $f \in \mathcal F $, all zeros of $f$ have multiplicity at least $k+2,$ and $f^{(k)}(z) \neq  b$ on $D,$ then $\mathcal F$ is normal on $D.$\\
\end{thmA}
Wang and Fang ~\cite{WF} gives an example to show that Theorem \ref{theorem1}, is not valid if all zeros of $f$ have multiplicity less than $k+2.$\\

By the ideas of shared values, M. Fang and L. Zalcman ~\cite{FZ, GPF} proved
\begin{thmB}
Suppose that $k$ is a positive integer and $b \neq 0$ be a finite complex number. Let $\mathcal F$ be a family of meromorphic function in a domain $D$, all of zeros of $f \in\mathcal F$ are of multiplicity at least $k+2.$ If for each $f, g \in \mathcal F, f$ and $g$ share $0, f^{(k)}$ and $g^{(k)}$ share $b$ IM in $D$, then $\mathcal F$ is normal in $D.$ 
\end{thmB}
In 2009, Y. Li and Y. Gu ~\cite{LG} proved the following result.
\begin{thmC}
Let $\mathcal F$ be a family of meromorphic functions defined in a domain $D.$ Let $k, n \geq k+2$ be positive integers and $b \neq 0$ be a finite complex number. If for each pair of functions $f, g \in \mathcal F,$ $(f^{n})^{(k)}$ and $ (g^{n})^{(k)}$ share $b$ in $D$, then $\mathcal F$ is normal in $D.$
\end{thmC}

Recently, releasing the condition that poles of $f(z)$  are of multiplicity at least $k+2$, J. Qui and T. Zhu~\cite{QZ}  proved the following.\\

\begin{thmD}\label{thm2}
Let  $k \geq 2$ be an integer and let $b$ be a non zero finite complex number. Let $\mathcal F$ be a family of meromorphic functions defined in a domain $D,$ such that for each $f \in \mathcal F$,  all zeros of $f(z)$ have multiplicity at least $k+2,$ and all zeros of $f^{(k)}(z)$ are multiple. If for each $f, g \in \mathcal F$, $f$ and $g$ share $b$ in $D,$ then $\mathcal F$ is normal in $D.$
\end{thmD}
It is natural to ask whether Theorem D.  can be improved by the idea of sharing a holomorphic function. In this paper we study this problem and obtain the following result.
\begin{theorem}\label{thm 3}
Let $d \geq 0$ and $k \geq 2d+2$  be two integers and let $h \not \equiv 0$ be a holomorphic function in $D, $ and multiplicity of its all zeros is at most $d.$ Let $\mathcal F$ be a family of meromorphic functions in a domain $D$. If for each $f \in \mathcal F,$ the multiplicity of all zeros of $f $ is at least $k+2d+2,$ and multiplicity of all zeros of $f^{(k)}$ is at least $2d+2.$ If for each pair of functions $f, g \in \mathcal F, f$ and $g$ share $h$ in $D,$ then $\mathcal F$ is normal in $D.$ 
\end{theorem}
\section{Some Lemmas}
 In order to prove our results we need  the following lemmas. The well known Zalcman Lemma is a very important tool in the study of normal families. The following is a new version due to  Zalcman ~\cite{Zalc} (also see  \cite{Zalc 1}, p. $814$).
\begin{lemma}\label{lemma 1} Let $\mathcal F$ be a family of meromorphic  functions in the unit disk  $\Delta$, with the property that for every function $f\in \mathcal F,$  the zeros of $f$ are of multiplicity at least $l$ and the poles of $f$ are of multiplicity at least $k$. If $\mathcal F$ is not normal at $z_0$ in $\Delta$, then for  $-l< \alpha <k$, there exist
\begin{enumerate}
\item{ a sequence of complex numbers $z_n \rightarrow z_0$, $|z_n|<r<1$},
\item{ a sequence of functions $f_n\in \mathcal F$, }
\item{ a sequence of positive numbers $\rho_n \rightarrow 0$},
\end{enumerate}
such that $g_n(\zeta)=\rho_n^{\alpha}f_n(z_n+\rho_n\zeta) $ converges to a non-constant meromorphic function $g$ on $\C$ with $g^{\#}(\zeta)\leq g^{\#}(0)=1$. Moreover, $g$ is of order at most two. Here, $g^{\#}(z)=\frac{|g'(z)|}{1+|g(z)|^2}$ is the spherical derivative of $g$.
\end{lemma}
\begin{lemma}\label{lemma 2}~\cite{PYZ}
Let $f(z)$ be a transcendental meromorphic function of finite order on $\C,$ and let $p(z) \not \equiv 0$ be a polynomial. Suppose that all zeros of $f(z)$ have multiplicity at least $k+1$, then  $f^{(k)}(z) - p(z)$ has infinitely many zeros.
\end{lemma}
\begin{lemma}\label{lemma 3}
Let $d \geq 0$ and $k \geq 2d+2$  be two integers and let $p(z) \not \equiv 0$ be a polynomial of degree at most $d.$ Let $f(z)$ is a non constant rational function and multiplicity of all  zeros of $f(z)$ is at least $k+2d+2$, and multiplicity of all zeros of $f^{(k)}(z)$ is at least $2d+2$. Then $f^{(k)}(z) - p(z)$ has at least two distict zeros and  $f^{(k)}(z) - p(z) \not \equiv 0.$
\end{lemma}
\begin{proof}
{Case 1.} Suppose that  $f^{(k)}(z)-p(z)$ has exactly one zero at $z_0$ with multiplicity $l.$ \\

{Case 1.1.} Suppose that $f$ is a non constant polynomial.\\

If  $f^{(k)}(z) - p(z) \equiv 0$, then $f(z)$ is a polynomial of degree at most $k+d,$  which contradicts with the fact that multiplicity of all zeros of $f(z)$  is at least $k+2d+2$, Hence  $f^{(k)}(z) - p(z) \not \equiv 0.$ Let

\begin{equation}\label{Ann eq b2}
f^{(k)}(z)-p(z)=K(z-z_0)^l
\end{equation}
where $K$ is a non-zero constant, $l$ is a positive integer. Because all zeros of $f^{(k)}(z)$ are of multiplicity at least $2d+2,$ we obtain $ l \geq 2d+2,$ then \\
\begin{equation}
f^{(k+d)}(z)-C = Kl(z-z_0)^{l-d},
\end{equation}
and $f^{(k+d+1)}(z) = Kl(l-d)(z-z_0)^{l-d-1}$.  This implies that $f^{(k+d-1)}$ has exactly one zero $z_0.$ So $f^{(k+d)}$ has only the same zero $z_0$ too. Hence $ f^{(k+d)}(z_0) = 0$, which contradicts with  $f^{(k+d)}(z_0) - C \neq 0.$ \\

 {Case 1.2.} Suppose that $f$ is a non-polynomial rational function.\\
Since $g(z)$ is  a rational function and not a polynomial, then obviously $f^{(k)}(z) - p(z) \not \equiv 0.$ Let\\
 \begin{equation}\label{eq1}
f^{(k)}(z) =\frac{A (z-\alpha_1)^{m_1}(z-\alpha_2)^{m_2}\ldots(z-\alpha_s)^{m_s}}{(z-\beta_1)^{n_1}(z-\beta_2)^{n_2}\ldots(z-\beta_t)^{n_t}},
\end{equation}
where $A$ is a non zero constant. Since all zeros of $f^{(k)}(z)$ are of multiplicity at least $2d+2, $ we find $m_i  \geq   2d+2 (i=1,2,\ldots ,s) , n_j \geq  k+1(j= 1,2, \ldots ,t).$\\
Let us define
\begin{equation}\label{eq1*}
\sum_{i=1}^s m_i = M \geq (2d+2)s\ \text{and}\ \sum_{j=1}^t n_j = N \geq (k+1)t \geq (2d+3)t.
\end{equation}
Differentiating both sides of (\ref{eq1}) step by step, we obtain
\begin{equation}\label{eq2}
f^{(k+d+1)}(z) =\frac{ (z-\alpha_1)^{m_1-d-1}(z-\alpha_2)^{m_2-d-1}\ldots(z-\alpha_s)^{m_s-d-1}g_1(z)}{(z-\beta_1)^{n_1+d+1}(z-\beta_2)^{n_2+d+1}\ldots(z-\beta_t){n_t+d+1}},
\end{equation}
where $g_1(z) =(M-N)\ldots(M-N-d)z^{(s+t-1)(d+1)} + a_{t-1}z^{t-1} + \ldots+ a_0 (a_{t-1}, \ldots a_0$ are constants).\\

Since $f^{(k)}(z) - p(z)$ has exactly one zero at $z_0.$ From (\ref{eq1}), we get
\begin{equation}\label{eq3}
f^{(k)}(z) = p(z) + \frac{B(z-z_0)^l}{(z-\beta_1)^{n_1}(z-\beta_2)^{n_2}\ldots(z-\beta_t)^{n_t}}
\end{equation}
Now we consider the following cases:\\
Case {1.2.1.} When $d\geq l$.
Differentiating both sides of (\ref{eq3}), $(d+1)$- times, we get
\begin{equation}\label{eq4}
f^{(k+d+1)}(z) = \frac{g_2(z)}{(z-\beta_1)^{n_1+d+1}(z-\beta_2)^{n_2+d+1}\ldots(z-\beta_t)^{n_t+d+1}}
\end{equation}
where $g_2(z) = (l - N) \ldots(l-N-d)z^{(d+1)t-(d-l+1)} +\ldots b_1 z+ b_0$.\\

From (\ref{eq1}) and(\ref{eq3}), we get $d+N = M$. This implies $M \geq N$.
Now from (\ref{eq2}) and (\ref{eq4}), we obtain\\
$M-(d+1)s \leq (d+1)t-(d-l+1) <(d+1)t$,\\
 
It follows that\\

$M<(d+1)(s+t)\leq(d+1)\left(\frac{M}{2d+2}+\frac{M}{2d+3}\right)< M.$\\

Which is a contradiction.\\

Case {1.2.2.} When $d<l.$\\

Differentiating both sides of (\ref{eq3}), $(d+1)$- times, we get
\begin{equation}\label{eq5}
f^{(k+d+1)}(z) = \frac{(z-z_0)^{l-d-1}g_3(z)}{(z-\beta_1)^{n_1+d+1}(z-\beta_2)^{n_2+d+1}\ldots(z-\beta_t)^{n_t+d+1}},
\end{equation}
where $g_3(z) = (l - N) \ldots(l-N-d)z^{(d+1)t} +\ldots c_1 z+ c_0$.\\

Differentiating (\ref{eq3}) $d$-times, we get $z_0$ is a zero of $f^{(k+d)}(z) - p^{(d)}(z)$, as $ p^{(d)}(z) \neq 0,$ then $z_0 \neq \alpha_i(i = 1,2,\ldots , s).$\\

\underline{Subcase {1.2.2.1}}   When $l<d+N$.\\

From (\ref{eq1}) and (\ref{eq3}), we get $M\geq N$. Since $z_0 \neq \alpha_i$ for any $ i\in \{1,2,\ldots, s\}$, therefore  from (\ref{eq2}) and (\ref{eq5}), we get\\ 

$M-(d+1)s \leq (d+1)t,$\\

It follows that\\

$M \leq \frac{4d+5}{4d+6}M < M,$
Which is a contradiction.\\

\underline{Subcase {1.2.2.2}} When $l \geq d+N.$\\

 If $M > N,$ then similar to the proof of Subcase 1.2.2.1, we get a contradiction. Thus $M\leq N.$ Since $z_0 \neq \alpha_i$ for any $ i\in \{1,2,\ldots, s\}$, then from (\ref{eq2}) and (\ref{eq5}),  we get\\

$l-d-1 \leq (d+1)(s+t-1),$\\

It follows that\\

$N\leq (d+1)\left(\frac{N}{2d+2}+\frac{N}{2d+3}\right) < N$, which is a contradiction.\\

{Case 2.} If  $f^{(k)}(z) - p(z)$ has no zero.  Similar to case 1, we obtain $f^{(k)}(z) - p(z)\not \equiv 0.$\\
Now put $l=0$ in (\ref{Ann eq b2}) and (\ref{eq3}), and  similar discussion to case 1, we get a contradiction.\\

Hence by case 1 and case 2, $f^{(k)}(z) - p(z)$ has at least two distinct zeros and $f^{(k)}(z) - p(z)\not \equiv 0. $
\end{proof}
\section{Proof of Main Result} 
\begin{proof}[Proof of Theorem \ref{thm 3}]  Since normality is a local property, it is enough to show that $\mathcal F$ is normal at each $z_0 \in D.$  we assume that
${D}=\Delta$. For each $z_0 \in D$, either $h(z_0) = 0$ or $h(z_0) \neq 0.$  Without loss of
generality, we may assume that $z_0=0$.\\

Case {1.} We first prove that $\mathcal F$ is normal at points $z$,  where $h(z) = 0$. By making standard normalization, we suppose that 
\begin{equation}\notag
h(z ) = z^l +a_{l+1}z^m+1+\ldots = z^lb(z),
\end{equation}  
where $l \geq 1, \  b(0) = 1,$ and $h(z) \neq 0$ when $ 0<|z|<1$. Let
\begin{equation}\notag
\mathcal F_1 := \{F_n: F_n(z) = \frac{f_n(z)}{z^l}, f\in \mathcal F\}.
\end{equation}
We shall prove that $\mathcal F_1$ is normal at $0.$\\
Suppose that $\mathcal F_1$ is not normal at $0,$ then by lemma \ref{lemma 1}, there exist $z_j \in \Delta$ tending to $0$, functions $F_j \in \mathcal F_1,$ positive numbers $\rho_j$ tending to $0$, such that 
\begin{equation}\label{eq 7}
g_n(\xi)= \rho_n^{-k}F_n(z_n + \rho_n \xi) \rightarrow g(\xi)
\end{equation}
locally uniformly on $\C$ with respect to the sherical  metric, where $g(\xi)$ is a non- constant meromorphic function on $\C$, whose order is at most 2. We distinguish two cases.\\

Case {1.1.} There exist a subsequence of $\frac{z_n}{\rho_n},$ we still denote the subsequence by $\frac{z_n}{\rho_n}$, such that $\frac{z_n}{\rho_n} \rightarrow \alpha$, where $\alpha$ is a finite complex number. Then,\\
\begin{equation}\notag
G_n(\xi) = \frac{f_n(\rho_n \xi)}{\rho_n^{k+l}}= (\rho_n\xi)^lF_n(z_n+\rho_n(\xi-\frac{z_n}{\rho_n})) \rightarrow \xi^lg(\xi-\alpha) = G(\xi)
\end{equation}  
spherically locally uniformly in $\C$. Then
\begin{equation}\label{eq 6}
G_n^{(k)}(\xi) -\frac{h(\rho_n\xi)}{\rho_n^l} = \frac{f_n^{(k)}(\rho_n\xi)- h(\rho_n\xi)}{\rho_n^l} \rightarrow G^{(k)}(\xi) -\xi^l
\end{equation}
spherically locally uniformly in $\C$.\\

Since for all $f\in \mathcal F$ multiplicity of all zeros of $f$ is at least $k+2d+2$, and multiplicity of all zeros of $f^{(k)}$ is at least $2d+2,$ which implies multiplicity of all zeros of $G$ is at least $k+2d+2$, and  by Hurwitz's theorem all zeros of $G^{(k)}$ is at least $2d+2,$  then by lemma (\ref{lemma 1}) and (\ref{lemma 2}), $ G^{(k)}(\xi) -\xi^l$ has at least two distinct zeros.\\

We claim that $ G^{(k)}(\xi) -\xi^l$ has just a unique zero.\\

Suppose that $\xi_0$ and $\xi_0^*$ are two distinct zeros of $ G^{(k)}(\xi) -\xi^l$, and choose $\delta (>0)$ small enough such that $D(\xi_0,\delta) \cap D(\xi_0^*,\delta) = \phi,$ where $D(\xi_0,\delta) = \{\xi | |\xi-\xi_0| < \delta\}$ and $D(\xi_0^*,\delta) =\{\xi | |\xi-\xi_0^*| < \delta\}$.\\

From (\ref{eq 6}), and by Hurwitz's theorem, there exists points  $\xi_n \in D(\xi_0,\delta), \xi_n^* \in D(\xi_0^*,\delta)$ such that for sufficiently large $n$
\begin{equation}\notag
f_n^{(k)}(\rho_n\xi_n)- h(\rho_n\xi_n) = 0,
\end{equation}
\begin{equation}\notag
f_n^{(k)}(\rho_n\xi_n^*)- h(\rho_n\xi_n^*) = 0.
\end{equation}
By the assumption that for each pair $f, g \in \mathcal F,$ $f^{(k)} $ and $g^{(k)}$ share $h$ in $D$, we know that for any integer $m$\\
\begin{equation}\notag
f_m^{(k)}(\rho_n\xi_n)- h(\rho_n\xi_n) = 0,
\end{equation}
\begin{equation}\notag
f_m^{(k)}(\rho_n\xi_n)- h(\rho_n\xi_n^*) = 0.
\end{equation}
we fix $m$ and note that $\rho_n\xi_n \rightarrow 0 , $ $\rho_n\xi_n^*\rightarrow 0,$ as $n\rightarrow \infty.$ From this we obtain 
\begin{equation}\notag
f_m^{(k)}(0)- h(0) = 0.
\end{equation}
Since the zeros of  $f_m^{(k)}(z)- h(z) = 0$ has no accumulation point, when $n$ is sufficiently large enough, we have
\begin{equation}\notag
\rho_n\xi_n = \rho_n\xi_n^* = 0.
\end{equation}

Hence\\

\begin{equation}\notag
\xi_n = \xi_n^* = 0.
\end{equation}
which contradicts with the fact that $\xi_n \in D(\xi_0,\delta), \xi_n^*\in D(\xi_0^* ,\delta) $ and $D(\xi_0,\delta) \cap D(\xi_0^*,\delta) = \phi$.\\

Case {1.2.}  There exist a subsequence of $\frac{z_n}{\rho_n},$ we still denote the subsequence by $\frac{z_n}{\rho_n}$, such that $\frac{z_n}{\rho_n} \rightarrow \infty$. By simple calculation we obtain,\\

\begin{align}\label{eq 1*}
F_n^{(k)}(z) & =\frac{f_n^{(k)}(z)}{z^l} - \sum_{m=1}^k C_k^m (z^l)^m\frac{F_n^{(k-m)}(z)}{z^l}\notag\\
& =\frac{f_n^{(k)}(z)}{z^l} - \sum_{m=1}^k C_m\frac{F_n^{(k-m)}(z)}{z^m},
\end{align}
where
\begin{equation}\notag
C_m = C_k^m l(l-1)\ldots (l-m+1),\ \text {when}\ m\leq l, \ \text{and}\ C_m = 0,\ \text{when}\ m>l.
\end{equation}
From (\ref{eq 1*}) and the identity $\rho_n^mg_n^{(k-m)} = F_n^{(k-m)}(z_n + \rho_n\xi)$, we get\\
\begin{align}\notag
g_n^{(k)}(\xi) & =  F_n^{(k)}(z_n + \rho_n\xi)\notag\\
& = \frac{f_n^{(k)}(z_n+\rho_n\xi)}{(z_n+\rho_n\xi)^l} - \sum_{m=1}^k C_m F_n^{(k-m)}(z)\frac{1}{(z_n+\rho_n\xi)^m}\notag \\
& =  \frac{f_n^{(k)}(z_n+\rho_n\xi)}{(z_n+\rho_n\xi)^l} - \sum_{m=1}^k C_m g_n^{(k-m)}(z)\frac{1}{(\frac{z_n}{\rho_n}+\xi)^m}.\\ \notag
\end{align}
Hence,\\
\begin{equation}\notag
\frac{f_n^{(k)}(z_n+\rho_n\xi)}{h(z_n+\rho_n\xi)} =\left( g_n^{(k)}(\xi) +  \sum_{m=1}^k g_n^{(k-m)}(\xi)\frac{C_m}{(\frac{z_n}{\rho_n}+\xi)^m}\right)\frac{1}{b(z_n+\rho_n\xi)}.
\end{equation}
 Thus, we have 

\begin{equation}\label{eq 8}
\frac{f_n^{(k)}(z_n+\rho_n\xi)-h(z_n+\rho_n\xi)}{h(z_n+\rho_n\xi)} \rightarrow g^{(k)}(\xi) -1,
\end{equation}

spherically uniformly on compact subset of $\C$ disjoint from the poles of $g.$\\

Since for all $f \in \mathcal F $, all zeros of  $f(z)$ have multiplicity at least $k+2d+2$, hence all zeros of  $g(\xi)$ have multiplicity at least $k+2d+2$. Noting that all zeros of $f^{(k)}(z)$  have multiplicity at least $2d+2$. By Hurwitz's theorem, all zeros of  $g^{(k)}(\xi)$ have multiplicity at least $2d+2.$ Thus   by lemma \ref{lemma 1} and \ref{lemma 2},  $ g^{(k)}(\xi) -1$ has at least two distinct zeros.\\

We claim that $ g^{(k)}(\xi) - 1$ has just a unique zero.\\

Suppose that $\xi_1$ and $\xi_1^*$ are two distinct zeros of $ g^{(k)}(\xi) - 1$, and choose $\delta (>0)$ small enough such that $D(\xi_1,\delta) \cap D(\xi_1^*,\delta) = \phi,$ where $D(\xi_1,\delta) = \{\xi | |\xi-\xi_1| < \delta\}$ and $D(\xi_1^*,\delta) =\{\xi | |\xi-\xi_1^*| < \delta\}$.\\

From (\ref{eq 8}), and by Hurwitz's theorem, there exists points  $\hat{\xi_n} \in D(\xi_1,\delta), \hat{\xi_n^*} \in D(\xi_1^*,\delta)$ such that for sufficiently large $n$
\begin{equation}\notag
f_n^{(k)}(z_n + \rho_n\hat{\xi_n})- h(z_n + \rho_n\hat{\xi_n}) = 0.
\end{equation}
\begin{equation}\notag
f_n^{(k)}(z_n + \rho_n\hat{\xi_n^*})- h(z_n + \rho_n\hat{\xi_n^*}) = 0.
\end{equation}

Similar to the proof of case 1, we get a contradiction. Hence $\mathcal F_1$ is normal at $0.$ It remains to prove that $\mathcal F$ is normal at 0.\\  

Since $\mathcal F_1$ is normal at $0$, then there exist $r >0$ and a subsequence  $F_{n_k}$ of $F_n$ such that $F_{n_k}$ converges sherically uniformally to a meromorphic function $F(z)$ or $\infty.$ in $\Delta_r$\\

Now we consider two cases:\\
case {i.} When $f_{n_k}(0) \neq 0$, for $k$ large enough. Then $F(0) = 0,$ then there exist $0<\delta<r $ such that $|F(z)| \geq1$ in $\Delta(0,\delta)$. Thus $|F_{n_k}|>\frac{1}{2}$  in $\Delta(0,\delta)$, for sufficiently large $k.$  Hence  $f_{n_k}$ is holomorphic in $\Delta(0,\delta.)$ Therefore ,\\
\begin{equation}\notag
\frac{1}{|f_{n_k}(z)|} = \frac{1}{|F_{n_k}(z)|}\frac{1}{z^l} \leq \frac{2^{l+1}}{\delta^m}, \ \text{for all} \  z \in \Delta(0, \frac{\delta}{2}).\\
\end{equation}

By the Maximum principle and Montel's theorem, $\mathcal F$ is normal at $0,$ and thus $\mathcal F$ is normal in $D.$\\

Case {ii.} When $f_{n_k}(0) = 0$, for $k$ large enough. Since the multiplicity of all zeros of $f\in \mathcal F$ is at least $k+2d+2,$ then $F(0) = 0. $   Hence, there exist $0<\rho<r$ such that $F(z)$ is holomorphic in $\Delta_\rho$. Hence $F_{n_k}$ converges  spherically locally uniformly to a holomorphic  function $F(z)$ in $\Delta_\rho$, hence, $f_{n_k}$ converges spherically locally uniformly to a holomorphic function $z^lF(z) $ in $\Delta_\rho.$ Hence $\mathcal F$ is normal at $0$, and thus $\mathcal F$ is normal in $D.$\\
  
Case {2.} Now we  prove that $\mathcal F$ is normal at points $z$,  where $h(z) \neq 0.$\\
Suppose that $\mathcal F$ is not normal at $z_0,$ then by lemma \ref{lemma 1}, there exist $z_j \in \Delta$ tending to $0$, functions $f_n\in \mathcal F,$ positive numbers $\rho_j$ tending to $z_0$, such that 
\begin{equation}\label{eq 7}
g_n(\xi)= \rho_n^{-k}f_n(z_n + \rho_n \xi) \rightarrow g(\xi)
\end{equation}
locally uniformly on $\C$ with respect to the sherical  metric, where $g(\xi)$ is a non- constant meromorphic function on $\C$, the multiplicity of all zeros of $g$ is at least $k+2d+2$, and multiplicity of all zeros of $g^{(k)}$ is at least $2d+2.$\\

 Hence by lemma \ref{lemma 2} and lemma \ref{lemma 3},  $g^{(k)}(\xi) - h(z_0)$ has at least two distinct zeros, and $g^{(k)}(\xi) - h(z_0)\not \equiv 0.$\\
 
Similar to the proof of case 1, we get a contradiction. Hence $\mathcal F$ is normal at $z_0.$\\
Since $z_0 $ is arbitrary, thus $\mathcal F$  is normal in $D.$
\end{proof}

\end{document}